\documentclass[12pt]{amsart}

\usepackage{amsmath}
\usepackage{amsfonts}
\usepackage{amssymb,enumerate}
\usepackage{amsthm}
\usepackage{fancyhdr}
\usepackage{hyperref}
\usepackage{xcolor}

\newtheorem{lemma}{Lemma}

\newtheorem{theorem}[lemma]{Theorem}

\theoremstyle{definition}

\theoremstyle{remark}
\newtheorem{rem}[lemma]{\bf Remark}
\newtheorem{ex}[lemma]{\bf Example}

\newcommand{\mm}{\mathfrak{m}}
\newcommand{\Tor}{\mathrm{Tor}}
\newcommand{\gr}{\mathrm{gr}}
\newcommand{\coker}{\mathrm{coker}}
\newcommand{\MM}{\mathcal{M}}
\newcommand{\NN}{\mathcal{N}}
\newcommand{\Image}{\mathrm{Image}}
\newcommand{\Lex}{\mathrm{Lex}}

\title{Consecutive cancellations in Tor modules over local rings}

\author{Alessio Sammartano}



\address{Department of Mathematics, Purdue University,  West Lafayette, IN 47907, USA}

\email{asammart@purdue.edu}

\begin{document}

\begin{abstract}
Let $M,N$ be finite modules over a Noetherian local ring $R$.
We show that the bigraded Hilbert series of 
$\gr (\Tor^R(M,N))$
 is obtained from that of $\Tor^{\gr(R)}(\gr(M),\gr(N))$
 by negative consecutive cancellations, 
thus extending a  theorem of Rossi and Sharifan.
\end{abstract}

\maketitle

\section{Introduction}

Given a module over a local or graded ring, a major problem in commutative algebra is to understand the behavior of its homological data.
It is often convenient to deal with ``approximations'' of the module with nicer properties,
and try to compare their invariants with those of the original module.
A well-known phenomenon is the fact that the Betti numbers increase
when passing to initial submodules with respect to term orders or weights (cf. \cite[Theorem 8.29]{MS}),
lex-segment submodules (cf. \cite{Bi,Hu,IP,Pa}), or associated graded modules with respect to filtrations (cf. \cite{F}). 
The idea of consecutive cancellations, introduced by Peeva \cite{Pe} in the graded context, 
arises in the attempt of describing these inequalities more precisely.
Peeva showed  that the graded Betti numbers of a homogeneous ideal $I$ in a polynomial ring are obtained from those of the associated lex-segment ideal $\Lex(I)$ by a sequence of \emph{zero consecutive cancellations},
i.e. cancellations from two Betti numbers $\beta_{i+1,j}(\Lex(I)), \beta_{i,j}(\Lex(I))$.
The theorem can be generalized to modules (cf. \cite[Theorem 1.1]{Pe} and the subsequent remark).

Now let $(R,\mm)$ be a local Noetherian ring and $\gr_\mm(R)$ its associated graded ring.
A common strategy used to study homological invariants over $R$ is to endow modules and complexes with filtrations and consider the associated graded objects,
thus taking advantage of the rich literature available for graded rings.
In this framework, 
Rossi and Sharifan \cite{RS1,RS} investigate  the relation between the minimal free resolution of an $R$-module $M$ and that of the associated graded module $\gr_\MM (M)$  over $\gr_\mm(R)$,
where $\MM$ is the $\mm$-adic filtration or more generally an $\mm$-stable filtration of $M$
 (see the next section for definitions).
Inspired by Peeva's result, 
they introduce the notion of \emph{negative consecutive cancellations}, i.e. cancellations from two Betti numbers $\beta_{i+1,a}(\gr_\MM (M)), \beta_{i,b}(\gr_\MM (M))$ with $a<b$.
They prove the following result for the case of a regular local ring:

\begin{theorem}[{\cite[Theorem 3.1]{RS}}]\label{RossiSharifanTheorem}
Let $(R,\mm)$ be a regular local ring and $\gr_\mm (R)$ its associated graded ring.
If $M$ is a finite $R$-module with an $\mm$-stable filtration $\MM$,
the  Betti numbers of $M$ are obtained from the graded Betti numbers of the $\gr_\mm (R)$-module $\gr_\MM (M)$ by a sequence of negative consecutive cancellations.
\end{theorem}

Their argument relies on explicit manipulations on the matrices of the differential 
 in a free $R$-resolution of $M$ obtained by lifting the minimal free $\gr_\mm (R)$-resolution of $\gr_\MM (M)$.
In this note we take an alternative approach, namely we employ a spectral sequence, introduced by Serre \cite{Se}, arising from a suitable filtered complex.
We provide the following generalization of Theorem \ref{RossiSharifanTheorem}, which holds for Tor modules over an arbitrary Noetherian local ring:

\begin{theorem}\label{Main}
Let $(R,\mm)$ be a Noetherian local ring, $\Bbbk=R/m$ its residue field, and $\gr_\mm (R)$ its associated graded ring.
Given two    finite $R$-modules $M,N$  with $\mm$-stable filtrations $\MM, \NN$,
there exist $\mm$-stable filtrations on each $\Tor_i^R(M,N)$ such that the Hilbert series
$$
\sum_{i,j}\dim_\Bbbk\gr (\Tor_i^R(M,N))_j z^it^j 
$$
 is obtained from the Hilbert series 
$$
\sum_{i,j}\dim_\Bbbk \Tor_i^{\gr_\mm (R)}(\gr_\MM (M),\gr_\NN (N))_j z^it^j
$$
by  negative consecutive cancellations, i.e. subtracting terms of the form $z^{i+1}t^a+z^{i}t^b$ with $a<b$.
\end{theorem}

\section{Proof of the main result}

We begin by providing the necessary background and definitions.
Let $(R,\mm)$  be a Noetherian local ring with residue field $\Bbbk=R/\mm$.
The ring $G= \gr_\mm(R)=\bigoplus_{n\geq 0} {\mm^j}/{\mm^{j+1}}$ is known as the \emph{associated graded ring} of $R$,
and it is a standard graded $\Bbbk$-algebra.
Geometrically, 
if $R$ is the local ring of a variety $\mathcal{V}$ at a point $\mathfrak{p}$,
then $G$ is the homogeneous coordinate ring of the tangent cone to $\mathcal{V}$ at $\mathfrak{p}$;
for this reason $G$ is sometimes called the tangent cone of $R$.

Let $M$ be a finitely generated $R$-module.
A descending filtration $\MM = \{M^j\}_{j\in \mathbb{N}}$ of $R$-submodules of $M$ is said to be \emph{$\mm$-stable} if the following conditions are satisfied:
 $M^0=M$, $\mm M^j \subseteq M^{j+1}$ for all $j\geq 0$, and $\mm M^j = M^{j+1}$ for sufficiently large $j$.
If $\MM$ is an $\mm$-stable filtration of $M$, we define the \emph{associated graded module of $M$ with respect to $\MM$} as $\gr_\MM (M) =\bigoplus_{j\geq 0} {M^j}/{M^{j+1}}$;
it is a finitely generated graded $G$-module.
The central example is that of the $\mm$-adic filtration $\MM = \{ \mm^j M \}_{j\in \mathbb{N}}$,
however it will be necessary to allow for more general ones.
In order to simplify the notation, sometimes we will just write $\gr (M)$ if the filtration is clear from the context.

Let $M,N$ be two $R$-modules with $\mm$-stable filtrations.
An $R$-linear map $f:M\rightarrow N$ such that $f(M^j)\subseteq N^j$ for all $j$ induces a homogeneous $G$-linear map $\gr(f):\gr(M)\rightarrow \gr(N)$.
In fact, $\gr(\cdot)$ is a functor from the category of $R$-modules with $\mm$-stable filtrations to the category of graded $G$-modules,
and in particular we can consider associated graded complexes of filtered complexes.
An important result due to Robbiano \cite{Ro} states that it is possible to ``lift'' a graded free $G$-resolution ${\bf G}$ of $\gr_\MM (M)$ to a free $R$-resolution ${\bf F}=\{F_i\}$ of $M$ together with 
$\mm$-stable filtrations $\{F^j_i\}_{j\in \mathbb{N}}$ on each free module $F_i$   such that
$\gr ({\bf F})={\bf G}$.
The lifted resolution ${\bf F}$ may not be minimal in general, even if we start from a minimal one, and the filtration  on $F_i$ may not be the $\mm$-adic one.
Compare also \cite[Theorem 1.8]{RS1}.

We refer to \cite{RS1}, \cite[Chapter 5]{E} for further background on filtered modules and to \cite[Appendix 3]{E} for background on spectral sequences.

\begin{proof}[Proof of Theorem \ref{Main}]
We analyze the convergence of the  spectral sequence  
$$
\Tor^{G}(\gr_\MM (M), \gr_\NN (N) ) \Rightarrow  \Tor^R(M,N). 
$$
The spectral sequence is constructed as follows.
Let $\MM=\{M^j\}_{j\in \mathbb{N}}$ and $\NN=\{N^j\}_{j\in \mathbb{N}}$ denote the given filtrations on $M,N$.
Let ${\bf F}=\{F_i\}$ be a free $R$-resolution of $M$ with $\mm$-stable filtrations $\mathcal{F}_i=\{F^j_i\}_{j\in \mathbb{Z}}$ on each $F_i$   such that
the associated graded complex $\gr ({\bf F})$ is a graded free $G$-resolution of $\gr_\MM (M)$.
Consider the complex  ${\bf L} =  {\bf F} \otimes_R N$ .
We equip each $L_i = F_i \otimes_R N$   with  filtrations $\mathcal{L}_i=\{L_i^j\}_{j\in\mathbb{N}}$ defined by
$$
L_i^j = \sum_{j_1+j_2=j} \Image \left( F_i^{j_1} \otimes_R N^{j_2} \rightarrow F_i \otimes_R N\right)
$$
where the maps in parentheses are obtained by tensoring the natural maps $ F_i^{j_1} \rightarrow F_i $ and $  N^{j_2} \rightarrow  N$.
Since $\mathcal{F}_i$ and $\MM$ are $\mm$-stable,  $\mathcal{L}_i$ is $\mm$-stable as well.
With this filtration we have the isomorphism of graded free complexes over $G$
$$
\gr({\bf L}) \cong \gr({\bf F}) \otimes_G \gr_\NN (N).
$$
By the Artin-Rees lemma, $\mathcal{L}_i$ in turn induces    an $\mm$-stable filtration on  $H_i({\bf L})=\Tor_i^R(M, N)$.

For the spectral sequence $\{^rE\}_{r\geq 1}$ of the filtered complex ${\bf L}$ we have
\begin{eqnarray*}
^1E_i^j & = & H_i(\gr({\bf L}))_j=H_i(\gr({\bf F})\otimes_G\gr_\MM (N))_j = \Tor_i^G(\gr_\MM (M), \gr_\NN (N))_j, \\
^\infty E_i^j & = & \gr( H_i({\bf L}))_j=\gr( H_i({\bf F}\otimes_R N))_j = \gr(\Tor_i^R(M, N))_j.
\end{eqnarray*}
The $\Bbbk$-vector space $^{r+1}E_i^j$ is a subquotient of $^rE_i^j$ for each $r$, hence the cancellation from $^1E_i^j$ to $^\infty E_i^j$ is the sum of  the cancellations from $^rE_i^j$ to $^{r+1}E_i^j$ for all $r\geq 1$.
In order to prove the theorem, it suffices to examine  cancellations in two consecutive pages $^rE, ^{r+1}E$ of the spectral sequence, so we fix $r$ for the rest of the proof.

Let $d$ denote the differential of ${\bf L}$.
Cycles and boundaries in the spectral sequence are given by
\begin{eqnarray*}
^{r+1}Z^j_i &=&\frac{ \{z \in L_i^j \, : \, dz \in L_{i-1}^{j+r+1}\}+L_{i}^{j+1} }{L_i^{j+1}+dL_{i-1}^j},\\
^{r+1}B^j_i &=&\frac{ (L^j_i \cap dL_{i-1}^{j-r})+L^{j+1}_i }{L^{j+1}_{i}+dL^j_{i-1}}.
\end{eqnarray*}
The inclusions  ${^rB_i^j} \subseteq {^{r+1}B_i^j}\subseteq {^{r+1}Z_i^j} \subseteq {^rZ_i^j}$
induce  maps of $\Bbbk$-vector spaces
$$
^rE_i^j = \frac{^rZ_i^j}{^rB_i^j} \overset{\pi_{i,j}}\longrightarrow \frac{^rZ_i^j}{^{r+1}B_i^j} \overset{\iota_{i,j}}\longleftarrow \frac{^{r+1}Z_i^j}{^{r+1}B_i^j} = {^{r+1}E_i^j}
$$
with $\pi_{i,j}$  surjective and $\iota_{i,j}$  injective,
so we have
$$
\dim_\Bbbk\, ^rE_i^j = \dim_\Bbbk  {^{r+1}E_i^j} + \dim_\Bbbk \ker (\pi_{i,j})+\dim_\Bbbk\coker( \iota_{i,j}).
$$
Therefore a cancellation from page $r$ to page $r+1$ can 
occur either in $\ker (\pi_{i,j})$ or in $\coker( \iota_{i,j})$.
We write down explicitly

\begin{eqnarray*}
\coker (\iota_{i,j})&=& \frac{^rZ_i^j}{^{r+1}Z_i^j}=\frac{\{z \in L^j_i \, : \, dz \in L^{j+r}_{i-1}\} }{\{z\in L^j_i \, : \, dz \in L^{j+r+1}_{i-1}\}+(L^{j+1}_i \cap \{z\in L^j_i \, : \, dz \in L^{j+r}_{i-1}\})},\\
\ker (\pi_{i-1,j+r}) &=& \frac{^{r+1}B_{i-1}^{j+r}}{^rB_{i-1}^{j+r}} = \frac{(L^{j+r}_{i-1}\cap dL^{j}_{i})+L^{j+r+1}_{i-1}}{(L^{j+r}_{i-1}\cap dL^{j+1}_{i})+L^{j+r+1}_{i-1}}.
\end{eqnarray*}
The differential $d$ induces a $\Bbbk$-linear map $\delta : \coker (\iota_{i,j}) \rightarrow \ker (\pi_{i-1,j+r})$
which sends $\overline{z}$ to $\overline{dz}$,
and from the two expressions above we obtain immediately that $\delta$ is bijective.
We conclude that cancellations in $\coker(\iota)$  in homological degree $i$ and internal degree $j$
correspond bijectively to cancellations in $\ker(\pi)$ in homological degree $i-1$ and internal degree $j+r$,
thus yielding  the negative consecutive cancellations at each page.
\end{proof}

\begin{rem}
We see from the proof that the degree $r=b-a$ of a negative consecutive  cancellation $z^{i+1}t^a+z^it^b$ corresponds to the page of the spectral sequence in which the cancellation occurs.
\end{rem}

We illustrate Theorem \ref{Main} with a couple of examples.

\begin{ex}
Let $R = \Bbbk[[X,Y]]$ be the formal power series ring with  $\mm=(X,Y)$ and let 
$M = R/(X^2-Y^3),  N = R/(X^2-Y^a)$ 
for some integer $a>3$. 
We compute
\begin{eqnarray*}
\Tor_0^R(M,N) &=& M\otimes_RN = R/(X^2-Y^3, X^2-Y^a),\\
\Tor_1^R(M,N) &=& \frac{(X^2-Y^3)\cap(X^2-Y^a)}{(X^2-Y^3)(X^2-Y^a)}=0.
\end{eqnarray*}
The associated graded ring of $R$ is $G= \Bbbk[x,y]$ where $x,y$ are the initial forms  of $X,Y$,
while the associated graded modules with respect to the $\mm$-adic filtrations are $\gr_\mm(M)=\gr_\mm(N)= G/(x^2)$.
We get
\begin{eqnarray*}
\Tor_0^G(\gr_\mm(M), \gr_\mm(N)) &=& \gr_\mm(M) \otimes_G \gr_\mm(N) = G/(x^2),\\
\Tor_1^G(\gr_\mm(M), \gr_\mm(N)) &=& \frac{(x^2)\cap(x^2)}{(x^2)(x^2)} = \frac{(x^2)}{(x^4)},
\end{eqnarray*}
hence their Hilbert series are respectively
$$
1 + 2\sum_{j\geq 1} t^j
\quad\mbox{ and }\quad
zt^2 + 2z\sum_{j\geq 3} t^j.
$$
Since the ideal $(X^2-Y^3, X^2-Y^a)$ is $\mm$-primary of colength 6 and the cancellations have negative degree by Theorem \ref{Main}, 
we conclude that we have precisely the cancellations $zt^2+ t^3$ and $2z t^j + 2t^{j+1}$ for all $j\geq 3$.
\end{ex}

Now we give an example of an infinite free resolution where no cancellation occurs.

\begin{ex}
Let $(R,\mm)$ be a Noetherian local hypersurface ring with associated graded ring $G \cong \Bbbk[x_1, \ldots, x_n]/(x_1^e)$ for some $e\geq 3$.
Suppose  $M$ is a factor ring of $R$ whose associated graded ring with respect to the $\mm$-adic filtration is of the form $\gr_\mm(M) \cong\Bbbk[x_1, \ldots, x_n]/I$ where 
$$
I = (x_1^d, x_1^{d-1}x_2, \ldots, x_1^{d-1}x_m)
$$
for some $m\leq n$ and $2\leq d<e$.
Since $(x_1^e)\subseteq I$ are stable monomial ideals, we can use \cite[Theorem 3]{IP} to determine
\begin{equation}\label{exampleNoCancellation}
\sum_{i,j}\dim_\Bbbk \Tor_i^{\gr_\mm (R)}(\gr_\mm (M),\Bbbk)_j z^it^j = \left( 1 + m z t^d +(m-1)z^2t^{d+1})\right)\sum_{k\geq 0} z^{2k}t^{ke}.
\end{equation}
In particular, we obtain the degrees of the non-zero graded components of $\Tor_i^{\gr_\mm (R)}(\gr_\mm (M),\Bbbk)$:
\begin{itemize}
\item if $i$ is odd then $\Tor_i^{\gr_\mm (R)}(\gr_\mm (M),\Bbbk)$ is concentrated in the single degree 
$ \frac{i-1}{2}e+d$;

\item if $i$ is even then $\Tor_i^{\gr_\mm (R)}(\gr_\mm (M),\Bbbk)$ is concentrated in  degrees
$\frac{i-2}{2}e+d+1$ and $\frac{i}{2}e$.
\end{itemize}
Since $d<e$, we conclude that the bigraded Hilbert series (\ref{exampleNoCancellation}) admits no negative consecutive cancellation.
By Theorem \ref{Main} we get 
$$
\gr (\Tor_i^R(M,\Bbbk))_j = \Tor_i^{\gr_\mm (R)}(\gr_\mm (M),\Bbbk)_j \quad \mbox{for all } i,j\geq 0.
$$
\end{ex}

\subsection*{Acknowledgements}
The author would like to thank his advisor Giulio Caviglia for suggesting this problem.
The author was supported by a grant of the Purdue Research Foundation.

\end{document}